\documentclass[11pt, reqno]{amsart}
\usepackage[dvipsnames,usenames]{color}
\usepackage[colorlinks=true, urlcolor=NavyBlue, linkcolor=NavyBlue, citecolor=NavyBlue]{hyperref}
\usepackage{cleveref}
\usepackage{graphicx}
\usepackage{epsfig}
\usepackage[latin1]{inputenc}
\usepackage{amsmath}
\usepackage{amsfonts}
\usepackage{amssymb}
\usepackage{amsthm}
\usepackage{amscd}
\usepackage{verbatim}
\usepackage{subfigure}
\usepackage{caption}
\usepackage{pinlabel}
\usepackage{stmaryrd}
\usepackage{enumerate, enumitem}
\usepackage{todonotes}
\usepackage{bm}
\usepackage{thmtools}
\usepackage{thm-restate}
\usepackage{lipsum}
\usepackage{setspace}
\usepackage{mathtools}
\usepackage[all]{xypic}
\usepackage[abs]{overpic}
\usepackage{color}
\usepackage[normalem]{ulem}
\usepackage[alphabetic,backrefs,msc-links]{amsrefs}
\usepackage{mathdots}
\usepackage{tikz-cd}

\allowdisplaybreaks

\usepackage{tikz}
\usetikzlibrary{arrows}
\usetikzlibrary{decorations.pathreplacing}

\usepackage{verbatim}
\usetikzlibrary{cd}
\usetikzlibrary{patterns}
\tikzset{taar/.style={double, double equal sign distance, -implies}}
\tikzset{amar/.style={->, dotted}}
\tikzset{dmar/.style={->, dashed}}
\tikzset{aar/.style={->, very thick}}

\newtheorem{theorem}{Theorem}[section]

\newtheorem{lemma}[theorem]{Lemma}
\newtheorem{proposition}[theorem]{Proposition}

\newtheorem{corollary}[theorem]{Corollary}

\newtheorem{conjecture}[theorem]{Conjecture}

\theoremstyle{definition}
\newtheorem{definition}[theorem]{Definition}

\theoremstyle{remark}
\newtheorem{remark}[theorem]{Remark}

\def\F{\mathbb{F}}

\def\Z{\mathbb{Z}}

\def\bbX{\mathbb{X}}

\def\Cone{\operatorname{Cone}}

\def\varep{\varepsilon}

\def\spinc{\textrm{Spin}^c}

\def\Cone{\operatorname{Cone}}

\def\HF {\mathit{HF}}
\newcommand\HFhat{\widehat{\HF}}

\def\CFK{\mathit{CFK}}

\def\spinc {{\operatorname{spin^c}}}

\def\s{\mathfrak s}

\newcommand{\tb}{\operatorname{tb}}
\newcommand{\rot}{\operatorname{rot}}

\newcommand\nuhat{\widehat{\nu}}
\newcommand\std{\mathrm{std}}

\newcommand{\ILtikzpic}[2][]{
\vcenter{\hbox{\begin{tikzpicture}[#1]
#2
\end{tikzpicture}}}
}

\newcommand{\drawover}[2][thick]{
\draw[line width=2mm,white] #2
\draw[#1] #2
}

\author[J.\ Hom]{Jennifer Hom}
\thanks{The first author was partially supported by NSF grants DMS-2104144 and DMS-2506400, and Georgia Tech's Elaine M. Hubbard Faculty Fellowship.}
\address {School of Mathematics, Georgia Institute of Technology, Atlanta, GA 30332}
\email{hom@math.gatech.edu}

\author[S. Wan]{Shunyu Wan}
\thanks{The second author was partially supported by Georgia Tech's postdoc funding.}
\address {School of Mathematics, Georgia Institute of Technology, Atlanta, GA 30332}
\email{swan48@gatech.edu}

\numberwithin{equation}{section}

\title[HF knot trace invariants, exotic $4$-manifolds, and symplectic obstructions]{Heegaard Floer knot trace invariants, exotic $4$-manifolds, and symplectic obstructions}

\begin{document}

\begin{abstract}
We show that the numerical invariants $\nu$ and $|\varepsilon|$ coming from knot Floer homology are knot $n$-trace invariants for any integer $n$, resolving the remaining case in \cite{HMP}. This extension allows us to construct new families of exotic pairs using Yasui patterns.  Moreover, by studying the invariant $\widehat{\nu}(K)=|\varepsilon(K)|(2\nu(K)-1)$, we give a new topological obstruction to a $4$-manifold being a strong symplectic filling of any contact structure on the boundary.
\end{abstract}

\maketitle

\section{Introduction}

The Heegaard Floer homology package \cite{ OS-knots, OS-3mfds, OS4manifold} contains a myriad of invariants of knots, 3-, and 4-manifolds. In this paper, we focus on the concordance invariants $\nu$ \cite{OS-rational} and $\varep$ \cite{Hom-cables}, with applications to knots traces, exotic 4-manifolds, and symplectic obstructions. Recall that the $n$-trace of a knot $K$, denoted $X_n(K)$, is the smooth 4-manifold obtained by attaching an $n$-framed $2$-handle along $K$ on the boundary of the standard four-ball.

We prove that $\nu(K)$ is an invariant of smooth, oriented knot traces:

\begin{theorem}\label{thm:trace}
For any integer $n$, if the oriented knot traces $X_n(K)$ and $X_n(K')$ are diffeomorphic, then $\nu(K)=\nu(K')$.
\end{theorem}

By Theorem 1.4 of \cite{HMP}, in order to prove Theorem \ref{thm:trace}, we need only show that if $n<0$ and $\nu(K)=0$ and $\nu(K')=1$, then $X_n(K)$ and $X_n(K')$ are not diffeomorphic. The proof of Theorem 1.4 in \cite{HMP} implicitly shows that $|\varep|$ is a $0$-trace invariant. We generalize this result as well and show that $|\varep|$ is also an $n$-trace invariant. 

\begin{theorem}\label{thm:trace epsilon}
For any integer $n$, if the oriented knot traces $X_n(K)$ and $X_n(K')$ are diffeomorphic, then $|\varep(K)|=|\varep(K')|$.
\end{theorem}

It is very interesting to point out the famous exotic trace pair $X_{-1}(-5_2)$ and $X_{-1}(P(3,-3,8))$ lies exactly in the category where $n<0$ and 
\begin{align*}
\nu(P(3,-3,8))&=0   &\nu(-5_2)&=1 \\
\varep(P(3,-3,8))&=0   &\varep(-5_2)&=1.
\end{align*}
 So now we can distinguish this pair using Heegaard Floer homology (or more specifically, knot Floer homology). This pair of exotic traces was first discovered by Akbulut \cite{Akbulut-exotic, Akbulut-fake} and has been extensively studied over time, from \cite{AkbulutLatveyev-exotic} to the recent work of Ren and Willis \cite{RenWillis}, who distinguished this pair combinatorially using skein lasagna modules. However, to the best of the authors' knowledge there is no purely Heegaard Floer proof of this fact before the present paper. 

Our extension not only reproves old results; using Yasui's satellite patterns we also give new infinite families of exotic pairs of knot traces. Let $P_{n,m}$ and $Q_{n,m}$ be the two satellite patterns shown in Figure \ref{fig:YasuiPQ}.

\begin{figure}
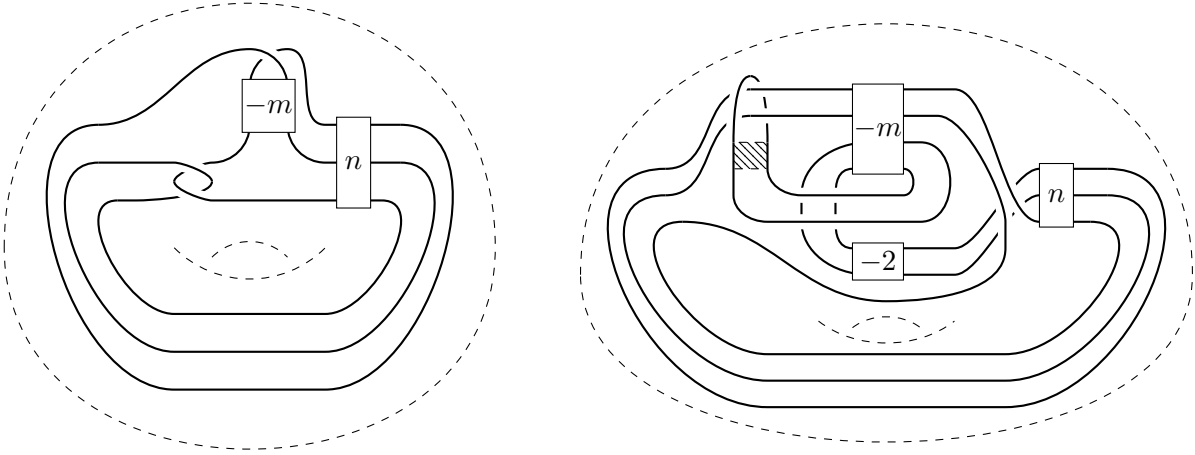

\[
\ILtikzpic[xscale=.5,yscale=.5]{
\draw[thick] (0.5,1) -- (1,1) to[out=0,in=-90,looseness=.5] (3,1.5);
\drawover[thick]{
    (2,1.5) to[out=-90,in=180,looseness=.5] (3,1) -- (7.5,1);
    }
\draw[thick] (2,1.5) to[out=90,in=180,looseness=.5] (3,2) to[out=0,in=-90] (4,3) -- (4,4) to[out=90,in=180] (5,5) to[out=0,in=180] (6,3) -- (8,3);
\drawover[thick]{
    (3,1.5) to[out=90,in=0,looseness=.5] (2,2) -- (0,2);
    }
\drawover[thick]{
    (8,2) -- (6,2) to[out=180,in=-90] (5,3) -- (5,4) to[out=90,in=0] (4,5) to[out=180,in=0] (0,3);
    }
\draw[fill=white] (3.8,4.2) rectangle (5.2,2.8);
\node at (4.5,3.5){$-m$};
\draw[fill=white] (6.3,3.2) rectangle (7.2,0.8);
\node at (6.75,2){$n$};
\draw[thick] (7.5,1) to[out=0,in=0] (6,-2) -- (2,-2) to[out=180,in=180] (0.5,1);
\draw[thick] (8,2) to[out=0,in=0] (6,-3) -- (2,-3) to[out=180,in=180] (0,2);
\draw[thick] (8,3) to[out=0,in=0] (6,-4) -- (2,-4) to[out=180,in=180] (0,3);
\draw[dashed] (3,-.5) to[out=45,in=135] (5,-.5);
\draw[dashed] (2,-.25) to[out=-45,in=-135] (6,-.25);
\draw[dashed] (-2.5,-.25) to[out=90,in=90,looseness=1.7] (10.5,-.25) to[out=-90,in=-90,looseness=1.4] (-2.5,-.25);
}
\quad
\ILtikzpic[xscale=.45,yscale=.35]{
\draw[thick] (0,5) to[out=0,in=-180] (2.5,8);
\draw[thick] (0,4) to[out=0,in=-180] (2.5,7);
\draw[thick] (2.5,8.5) to[out=0,in=90,looseness=.5] (3,6);
\drawover[thick]{
    (2.5,8.5) to[out=180,in=90,looseness=.5] (2,6);
    }
\draw[thick] (6.5,6) -- (6,6) to[out=180,in=90] (4,4) -- (4,3) to[out=-90,in=180] (6,1) -- (8.5,1) to[out=0,in=180,looseness=.5] (11,4) -- (13,4);
\draw[thick] (6.5,5) -- (5.5,5) to[out=180,in=90] (5,4) -- (5,3) to[out=-90,in=180] (5.5,2) -- (8.5,2) to[out=0,in=180,looseness=.5] (11,5) -- (13,5);
\drawover[thick]{
    (2.5,8) -- (8.5,8) to[out=0,in=180,looseness=.5] (11,3) -- (12.5,3);
    }
\drawover[thick]{
    (2.5,7) -- (8,7) to[out=0,in=90,looseness=.5] (10,3) -- (10,2) to[out=-90,in=0] (6.5,0) to[out=180,in=0] (0.5,3);
    }
\drawover[thick]{
    (6,6) -- (7.5,6) to[out=0,in=0] (7.5,3) -- (3,3) to[out=180,in=-90] (2,4) -- (2,6);
    }
\drawover[thick]{
    (6,5) -- (7,5) to[out=0,in=0] (7,4) -- (4,4) to[out=180,in=-90] (3,5) -- (3,6);
    }
\draw[fill=white] (5.5,8.2) rectangle (7,4.8);
\node at (6.25,6.5){$-m$};
\draw[fill=white] (5.5,2.2) rectangle (7,0.8);
\node at (6.25,1.5){$-2$};
\draw[fill=white] (11,5.2) rectangle (12,2.8);
\node at (11.5,4){$n$};
\draw[dashed,pattern=north west lines] (2,5) rectangle (3,6);
\draw[thick] (12.5,3) to[out=0,in=0] (10,-2) -- (3,-2) to[out=180,in=180] (0.5,3);
\draw[thick] (13,4) to[out=0,in=0] (10,-3) -- (3,-3) to[out=180,in=180] (0,4);
\draw[thick] (13,5) to[out=0,in=0] (10,-4) -- (3,-4) to[out=180,in=180] (0,5);
\draw[dashed] (5.5,-1) to[out=45,in=135] (7.5,-1);
\draw[dashed] (4.5,-.75) to[out=-45,in=-135] (8.5,-.75);
\draw[dashed] (-2.5,1) to[out=90,in=90,looseness=1.8] (15.5,1) to[out=-90,in=-90,looseness=1.2] (-2.5,1);
}
\]
\caption{Left, the satellite patterns $P_{n,m}$ and right, the pattern $Q_{n,m}$. The boxes denote full twists. The shaded band gives rise to a concordance between $Q_{n,m}$ and the identity pattern. Figure taken from \cite[Figure 3]{RenWillis} with permission from the authors.}
\label{fig:YasuiPQ}
\end{figure}

\begin{theorem}\label{thm:exotic traces}
Let $K \subset S^3$ be a knot and let $m$ be any non-negative integer. 
\begin{enumerate}
    \item If $\varep(K) = 0$, then $X_n(P_{n,m}(K))$ and $X_n(Q_{n,m}(K))$ form an exotic pair for any integer $n < 0$. Moreover, $X_n(P_{n,0}(K))$ and $X_n(Q_{n,0}(K))$ form an exotic pair for any integer $n \neq 0$.
    \item\label{it:exotic traces2} If $\varep(K) = 1$, then $X_n(P_{n,m}(K))$ and $X_n(Q_{n,m}(K))$ form an exotic pair for any integer $n < 2\tau(K)$.  
    \end{enumerate}
\end{theorem} 

Note that $X_{-1}(-5_2)$ and $X_{-1}(P(3,-3,8))$ are the special case of the above theorem when $K$ is the unknot,  $n=-1$, {and $m=0$}.

\begin{remark}
We make the following observations:
\begin{enumerate}
\item In \cite[Theorem 4.1]{Yasui}, Yasui shows that if there is a Legendrian representative $\mathcal{K}$  of $K$ in the standard tight $S^3$ with 
\[\tb(\mathcal{K})+|\rot(\mathcal{K})|=2g_4(K)-1, \quad n\leq \tb(\mathcal{K}), \quad m\geq 0,\]
then $X_n(P_{n,m}(K))$ and $X_n(Q_{n,m}(K))$ form an exotic pair. Note that according to \cite[Theorem 1]{Plamenevskaya04} we always have $\tb(\mathcal{K})+|\rot(\mathcal{K})|\leq 2\tau(K) -1 \leq 2g_4(K)-1$. Hence, if $K$ admits a Legendrian representative $\mathcal{K}$ satisfying Yasui's condition then $\tau(K)=g_4(K)$ and $n\leq \tb(\mathcal{K}) < 2\tau(K)$. Thus, Yasui's families with $g_4(K)=0$ are covered by the first part of our Theorem \ref{thm:exotic traces}, and families with $g_4(K) > 0$ are covered by the second part (since $\tau(K)=g_4(K)>0$ implies $\varep(K)=1$ \cite{Hom-cables} ). In particular, we fully recover Yasui's families.

\item In \cite[Theorem 1.3]{RenWillis}, Ren and Willis show that if a knot $K$ has a slice disk $\Sigma$ in $k\mathbb{CP}^2 \backslash \operatorname{int}(B^4)$ with
$$s(K)=|[\Sigma]|-[\Sigma]^2, \quad n<-[\Sigma]^2, \quad m\geq 0,$$
where $s(K)$ is the Rasmussen $s$-invariant \cite{Rasmussen-s}, then $X_n(P_{n,m}(K))$ and $X_n(Q_{n,m}(K))$ form an exotic pair. We refer the reader to their paper for the exact definition for each term above, but we observe that in the special case when $K$ is slice, their results only show $X_n(P_{n,0}(K))$ and $X_n(Q_{n,0}(K))$ form an exotic pair for $n <0 $, but our theorem also covers the situation when $n>0$. 

\item According to the construction from \cite{HMP}, these families of exotic knot traces can give rise to families of exotic Mazur
manifolds. Furthermore, our Theorem \ref{thm:exotic traces}\eqref{it:exotic traces2} strengthens \cite[Theorem 2.11]{HMP}, which states that for any knot $K \subset S^3$ with $\varep(K)=1$ and $\tau(K)>0$, and any integer $n \leq 0$, the knot traces $X_n(P_{n,0}(K))$ and $X_n(Q_{n,0}(K))$ form an exotic pair.
\end{enumerate}
\end{remark}

We make the following conjecture: 

\begin{conjecture}
Let $K$ be any knot in $S^3$, and $m, n$ integers. Then $X_n(P_{n,m}(K))$ and $X_n(Q_{n,m}(K))$ form an exotic pair, except when $K$ is the unknot and $m=n=0$.
\end{conjecture}
\noindent Note that when $K$ is the unknot and $m=n=0$, both $P_{n,m}(K)$ and $Q_{n,m}(K)$ are unknotted.

The invariants $\varep$ and $\nu$ play a key role in various places in the Heegaard Floer homology package, and it is convenient to make the following definition:

\begin{definition}
Given a knot $K$ in $S^3$, let $\widehat{\nu}(K)=|\varep(K)|(2\nu(K)-1)$
\end{definition}

As a direct corollary of Theorem \ref{thm:trace} and \ref{thm:trace epsilon}, $\widehat{\nu}$ is also a knot trace invariant. 

\begin{corollary}\label{cor:nuhat}
For any integer $n$, if the oriented knot traces $X_n(K)$ and $X_n(K')$ are diffeomorphic, then $\widehat{\nu}(K)=\widehat{\nu}(K')$.
\end{corollary}

The concordance invariant $\widehat{\nu}$ above is the same as the invariant $\widehat{\nu}$ introduced by Baldwin and Sivek \cite{BS-instantonconcordance} (see Lemma \ref{lem:same nu hat}),  and it has a close relationship with immersed curves and the rank of $\HFhat(S^3_n(K))$ \cite[Proposition 15]{HANSELMANHF} \cite[Proposition 4.2]{BS-Lspaceknottrace}; for more detailed discussion about $\widehat{\nu}$ see \cite[Section 10]{BS-instantonconcordance}. As a direct consequence of $\widehat{\nu}$ being an $n$-trace invariant, we give a generalization of \cite[Theorem 4.6]{BS-Lspaceknottrace}.

\begin{theorem}\label{thm:dimHF}
    Fix an integer $n\in \mathbb{Z}$ and a knot $K \subset S^3$. Given a nonzero rational number $r\in \mathbb{Q}$, the dimension \[\operatorname{dim}\HFhat(S^3_r(K))\] is completely determined by the oriented diffeomorphism type of $X_n(K)$, meaning that if $X_n(K) \cong X_n(K')$ then $\operatorname{dim}\HFhat(S^3_r(K))=\operatorname{dim}\HFhat(S^3_r(K'))$.
\end{theorem}

Proposition 10.1 of \cite{BS-instantonconcordance} shows that $\nuhat(K)$ gives us information on the $3$-dimensional level, playing a key role in determining the dimension of $\HFhat(S^3_r(K))$, while Corollary \ref{cor:nuhat} and Theorems \ref{thm:vanishing HF when summing} and \ref{thm:vanishing HF for all spinc} below show that $\nuhat$ also gives us important information on the $4$-dimensional level.

\begin{theorem}\label{thm:vanishing HF when summing}
Let $K$ be a knot in $S^3$, and 
\[
F_{X_n} \colon \HFhat(S^3) \longrightarrow \HFhat(S^3_n(K))
\]
be the Heegaard Floer map induced by the $n$-trace $X_n(K)$ summing over all $\spinc$ structures on $X_n$. Then:
\begin{itemize}
  \item If $\widehat{\nu}(K) \neq 0$, the map $F_{X_n}$ is trivial if and only if $n \geq \widehat{\nu}(K)$.
   \item If $\widehat{\nu}(K) = 0$, the map $F_{X_n}$ is trivial if and only if $n=-1$ or $n \geq 1$.  
   \end{itemize}
\end{theorem}

\begin{theorem}\label{thm:vanishing HF for all spinc}
Let $K$ be a knot in $S^3$, and 
\[
F_{{X_n},\mathfrak{s}} \colon \HFhat(S^3) \longrightarrow \HFhat(S^3_n(K))
\]
be the Heegaard Floer map induced by the $n$-trace $X_n(K)$ with some $\spinc$ structure $\mathfrak{s}$ on $X_n$. Then:
\begin{itemize}
    \item If $\widehat{\nu}(K) \neq 0$, the map $F_{{X_n},\mathfrak{s}}$ is trivial for all $\mathfrak{s}$ if and only if $n \geq \widehat{\nu}(K)$.
    \item If $\widehat{\nu}(K) = 0$, the map $F_{{X_n},\mathfrak{s}}$ is trivial for all $\mathfrak{s}$ if and only if $n \geq 1$.  
    \end{itemize}
\end{theorem}

As an application of the vanishing result, we give a new topological obstruction for a $4$-manifold being a strong symplectic filling of its boundary. We start with the result for knot traces. 

\begin{theorem}\label{thm:filling for traces}
Let $K$ be a knot in $S^3$. Then for any integer $n \geq \widehat{\nu}(K)+0^{\widehat{\nu}(K)}$ and any contact structure $\xi$ on $S^3_n(K)$, the trace $X_n(K)$ can not be a strong symplectic filling of $(S^3_n(K),\xi)$. 
\end{theorem}

\begin{remark}
We make the following remarks:
\begin{enumerate}

\item Using the convention that $0^0=1$ the condition $n \geq \widehat{\nu}(K)+0^{\widehat{\nu}(K)}$ covers both $\widehat{\nu}(K)\neq 0$ and $\widehat{\nu}(K)=0$.

\item It was known by the adjunction inequality for symplectic manifolds \cite{LM98, FS95, OS00} that if $n\geq 2g_4(K)-1$, where $2g_4(K)$ is the $4$-ball genus of the knot, then $W_n(K)$ cannot be a symplectic filling of $S_n^3(K)$. So the theorem is particular interesting for $\widehat{\nu
}(K)<2g_4(K)-1$; for example the negative torus knot $T_{2,-2m-1}$ has $\widehat{\nu}(T_{2,-2m-1})=-{2}m+1$ and $2g_4(T_{2,-2m-1})-1=2m-1$.
\item When $n\neq 0$, the manifold $S_n^3(K)$ is a rational homology sphere, so a weak symplectic filling on $W_n$ is actually strong \cite{OO09}. Hence, for example, Theorem \ref{thm:filling for traces} above implies that for any $n\geq -{2}m+1$ and $n\neq 0$, the trace $W_n(T_{2,-2m-1)}$ cannot even be a weak symplectic filling.
\item The above theorem should be compared with the Stein framing number defined in \cite{MPV20} and their Theorem 1.3.
\item The above theorem is also a generalization of \cite[Theorem 1.4]{HomLidman19}, which states that if $d(S^3_1(K))=0$, then for $n > 0$ and any contact structure $\xi$ on $S^3_n(K)$, the trace $X_n(K)$ is not a symplectic filling of $S^3(K), \xi)$. Indeed, we have that $d(S^3_1(K))=0$ implies $\tau(K)\leq 0$ \cite{HomWu}, hence $\widehat{\nu}(K)+0^{\widehat{\nu}(K)} \leq 1$ always.
\end{enumerate}
\end{remark}

Theorem \ref{thm:filling for traces} above for trace fillings can also be generalized to the $4$-manifold obtained by attaching multiple $2$-handles to the boundary of a $4$-ball. 

\begin{theorem}\label{thm:filling for 2 handles}
    Let $\mathbf{L}=(L_1,\cdots,L_k)$ be a $k$-component link in $S^3$, and let $\mathbf{n}=(n_1,\cdots,n_k)$ be a $k$-tuple of integers representing the framing of each link component.   If there is some $1\leq i \leq k$ such that $n_i \geq \widehat{\nu}(L_i)+0^{\widehat{\nu}(L_i)}$ and $n_i\neq 0$ then for any contact structure on $S^3_{\mathbf{n}}(\mathbf{L})$, the $4$-manifold $X_{\mathbf{n}}(\mathbf{L})$ obtained by attaching $2$-handles on $\mathbf{L}$ with framing $\mathbf{n}$ can not be a strong symplectic filling of $S^3_{\mathbf{n}}(\mathbf{L})$. 
\end{theorem}

Throughout the paper, we will consider Heegaard Floer homology with coefficients in $\mathbb{F} = \mathbb{Z}/2\mathbb{Z}$, and we will use $\cong$ to indicate orientation-preserving diffeomorphism.

\section*{Organization}
In Section \ref{sec:mapcone}, we review the mapping cone construction of \cite{OS-integer}, and use it to prove Theorems \ref{thm:trace}, \ref{thm:trace epsilon}, \ref{thm:vanishing HF when summing}, and \ref{thm:vanishing HF for all spinc}. We  consider exotic traces in Section \ref{sec:exotictraces}, where we prove Theorem \ref{thm:exotic traces}; a key ingredient is work of Bodish \cite{Bodish}, who uses immersed curves to compute $\tau$ and $\varepsilon$ for certain generalized Mazur satellites. We show that our definition of $\nuhat$ agrees with that of Baldwin-Sivek in Section \ref{sec:BS}, where we also prove Theorem \ref{thm:dimHF}. Lastly, we obstruct symplectic fillings and prove Theorems \ref{thm:filling for traces} and \ref{thm:filling for 2 handles} in Section \ref{sec:symplecticfillings}.

\section*{Acknowledgements}
The authors thank Tye Lidman, JungHwan Park, Qiuyu Ren, and Fan Ye for helpful conversations, and Tom Mark for comments on an earlier draft.
\section{The mapping cone and cobordism maps}\label{sec:mapcone}
Our main tool will be the mapping cone formula of \cite{OS-integer}, which we briefly review here, primarily to establish notation, which agrees with the notation of \cite{HMP}. See also \cite[Section 2]{HKL}.

We consider the case of integer surgery on a knot $K$ in $S^3$. We will work with the hat-flavor of Heegaard Floer homology. Let $C(K) = \CFK^\infty(K)$, which as a vector space decomposes as $C(K) = \bigoplus_{i, j\in \Z} C(i,j)$. For a subset $X \subseteq \Z^2$, let 
\[ CX = \bigoplus_{(i,j) \in X} C(i,j), \]
and
\begin{align*}
	A_s &= C\{ \max\{i, j-s\} =0 \} \\
	B &= C\{i=0\}.
\end{align*}
Since our ambient manifold is $S^3$, we have that $H_*(B) = \HFhat(S^3) = \F$.
There are maps 
\begin{align*}
	v_s \colon A_s \to B \\
	h_s \colon A_s \to B 
\end{align*}
where $v_s$ consists of quotienting by $C\{ i < 0, j=s\}$ and $h_s$ consists of quotienting by $C\{ i=0, j<s\}$ followed by a homotopy equivalence between $C\{j=s\}$ and $C\{i=0\}$.

Consider the map 
\[ D_n \colon \bigoplus_{s \in \Z} A_s \to \bigoplus_{s \in \Z} B_s \]
where each $B_s=B$ and for $(s,x) \in \bigoplus_{s \in \Z} A_s$ (where the first entry indicates the summand in which $x$ lies),
\[ D_n(s,x) = (s, v_s(x)) + (s+n, h_s(x)). \]

We denote the mapping cone of $D_n$ by $\bbX_n$, and the main result of \cite{OS-integer} states that the homology of $\bbX_n$ is isomorphic to $\HFhat(S^3_n(K))$. Moreover, the map induced by inclusion of $B_s$ into $\bbX_n$ corresponds to the map
\[F_{X_n, s} \colon \HFhat(S^3) \to \HFhat(S^3_n(K)) \]
induced by the 2-handle cobordism
\[ X_n \colon S^3 \to S^3_n(K) \]
equipped with the $\spinc$-structure $\s_s$ characterized by the property
\[ \langle c_1(\s_s), \sigma \rangle + n = 2s \]
where $\sigma$ is a generator of $H_2(X_n; \Z)$. If $X_n$ is clear from context, we may write simply $F_{s}$ rather than $F_{X_n,s}$.

Since we are working with the hat-flavor, we may pass to homology before taking the mapping cone. Note that since $H_*(B_s) = \F$, the maps
\[ v_{s,*}, h_{s,*} \colon H_*(A_s) \to H_*(B) \]
are either surjective or trivial. In particular, the invariant $\nu(K)$ from \cite[Definition 9.1]{OS-rational} determines whether these maps are trivial or nontrivial, as follows:
\begin{enumerate}
	\item $v_{s,*}$ is surjective if and only if $s \geq \nu(K)$, and 
	\item $h_{s,*}$ is surjective if and only if $s \leq -\nu(K)$. 
\end{enumerate}
Throughout, we will denote the generator of $H_*(B_s) = \F$ by $y_s$.

We will also be interested in the invariant $\varep(K)$ from \cite{Hom-conc}, which plays a key role in the following lemma:
\begin{lemma}[{\cite[Lemma 4.1]{HMP}}]\label{lem:HMP}
The maps $v_s, h_s \colon A_s \to B$ induce the same nontrivial map in homology if and only if $s = \varepsilon(K) = 0$.
\end{lemma}

Let $-K$ be the reverse of the mirror of $K$. We recall the relationships between $\tau$ \cite{OS-4ball}, $\nu$ \cite{OS-rational}, and $\varep$ \cite{Hom-conc}:
\begin{enumerate}
	\item $\tau(-K) = -\tau(K)$ and $\varep(-K) = -\varep(K)$,
	\item $\nu(K) = \tau(K)$ or $\tau(K)+1$,
	\item \begin{enumerate}
		\item  if $\nu(K) = \tau(K)+1$, then $\varep(K) = -1$,
		\item if $\nu(-K) = \tau(-K)+1$, then $\varep(K) = 1$, and
		\item if $\nu(K)=\tau(K)$ and $\nu(-K) = \tau(-K)$, then $\varep(K)=0$.
		\end{enumerate}
\end{enumerate}
These tools lead us to the following proposition:

\begin{proposition}\label{prop:nu01}
Let $n<0$.
\begin{enumerate}
	\item\label{it:nu1} If $\nu(K)=1$, then $F_{X_n, s} \neq 0$ if and only if $n \leq s \leq 0$, and $\bigoplus_s F_{X_n, s}$ has rank $|n|+1$.
	\item\label{it:nu0ep0} If $\nu(K)=0$ and $\varepsilon(K) = 0$, then $F_{X_n, s} \neq 0$ if and only if $n \leq s \leq 0$, and $\bigoplus_s F_{X_n, s}$ has rank $|n|$.
	\item\label{it:nu0ep1} If $\nu(K)=0$ and $|\varepsilon(K)| = 1$, then $F_{X_n, s} \neq 0$ if and only if $n+1 \leq s \leq -1$, and $\bigoplus_s F_{X_n, s}$ has rank $|n|-1$.
\end{enumerate}
\end{proposition}

\begin{remark}
Above, when we consider the rank of the map $\bigoplus_s F_{X_n, s}$, we mean the map
\begin{align*}
	\bigoplus_s F_{X_n, s} \colon \bigoplus_s \HFhat(S^3) &\to \HFhat(S^3_n(K)), \\
							(x_s)_{s \in \Z} &\mapsto \sum_s F_{X_n,s}(x_s)
\end{align*}
which is different from the map
\begin{align*}
	 F_{X_n} \colon \HFhat(S^3) &\to  \HFhat(S^3_n(K)) \\
	 	x &\mapsto \sum_s F_{X_n,s}(x)
\end{align*}
obtained by summing over all $\spinc$-structures on $X_n$.
\end{remark}

\begin{proof}
\eqref{it:nu1} Let $\nu(K)=1$. Recall that since $\nu(K)=1$, the map $v_{s,*}$ is surjective if and only if $s \geq 1$, and  $h_{s,*}$ is surjective if and only if $s \leq -1$. 

We first consider $F_{X_n, s}$ when $s \equiv 0$ mod $n$. The relevant part of the mapping cone formula is shown below:
\begin{equation}\label{eq:nu1s0}
\begin{tikzcd}
	& {H_*(A_{-2|n|})} & {H_*(A_{-|n|})} & {H_*(A_0)} & {H_*(A_{|n|})} & {H_*(A_{2|n|})} \\
	\dots &&&&& \dots \\
	& {H_*(B_{-2|n|})} & {\bm{H_*(B_{-|n|})}} & {\bm{H_*(B_0)}} & {H_*(B_{|n|})}
	\arrow["0", from=1-2, to=3-2]
	\arrow[two heads, from=1-3, to=3-2]
	\arrow["0", from=1-3, to=3-3]
	\arrow["0"', from=1-4, to=3-3]
	\arrow["0", from=1-4, to=3-4]
	\arrow["0"', from=1-5, to=3-4]
	\arrow[two heads, from=1-5, to=3-5]
	\arrow["0"', from=1-6, to=3-5]
\end{tikzcd}
\end{equation}
For $s \equiv 0$ mod $n$, $s \leq -2|n|$, we see that $H_*(B_s)$ is in the image of $D_{n,*}(H_*(A_{s+|n|})$, implying that for such $s$, the map $F_{X_n, s}$ is trivial. Similarly, for $s \equiv 0$ mod $n$, $s \geq |n|$, we see that $H_*(B_s)$ is in the image of $D_{n,*}(H_*(A_{s}))$, implying that for such $s$, the map $F_{X_n, s}$ is trivial.

Lastly, we consider $s=-|n|$ and $s=0$. Since the maps from $H_*(A_{-|n|})$ and $H_*(A_0)$ to $H_*(B_{-|n|})$ are both zero, it follows that $F_{-|n|}$ is nontrivial, and similarly for $F_0$. Furthermore, we observe that $y_{-|n|}$ and $y_0$ (the generators of $H_*(B_{-|n|})$ and $H_*(B_{0})$ respectively) represent distinct elements in $H_*(\Cone(D_n))$. Hence it follows that the image of $F_{-|n|}$ and $F_0$ are distinct in $\HFhat(S^3_n(K))$.
 
We now consider $F_{X_n, s}$ when $s \not\equiv 0$ mod $n$. The relevant part of the mapping cone is shown below, where $i=1, \dots, {|n|}-1$:
\begin{equation}\label{eq:nu1i}
\begin{tikzcd}
	& {H_*(A_{-2|n|+i})} & {H_*(A_{-|n|+i})} & {H_*(A_i)} & {H_*(A_{|n|+i})} & {H_*(A_{2|n|+i})} \\
	 \dots &&&&& \dots \\
	& {H_*(B_{-2|n|+i})} & {\bm{H_*(B_{-|n|+i})}} & {H_*(B_i)} & {H_*(B_{|n|+i})}
	\arrow["0", from=1-2, to=3-2]
	\arrow[two heads, from=1-3, to=3-2]
	\arrow["0", from=1-3, to=3-3]
	\arrow["0"', from=1-4, to=3-3]
	\arrow[two heads, from=1-4, to=3-4]
	\arrow["0"', from=1-5, to=3-4]
	\arrow[two heads, from=1-5, to=3-5]
	\arrow["0"', from=1-6, to=3-5]
\end{tikzcd}
\end{equation}

For $s \not\equiv 0$ mod $n$, $s \leq -2|n|+i$, we see that $H_*(B_s)$ is in the image of $D_{n,*}(H_*(A_{s+|n|}))$, implying that for such $s$, the map $F_{X_n, s}$ is trivial. Similarly, for $s \not\equiv 0$ mod $n$, $s \geq i$, we see that $H_*(B_s)$ is in the image of $D_{n,*}(H_*(A_{s})$, implying that for such $s$, the map $F_{X_n, s}$ is trivial.

When $s=-|n|+i$, we see that $y_{-|n|+i}$, the generator of $H_*(B_{-|n|+i})$, represents a nontrivial element in $H_*(\Cone(D_n))$, since $v_{-|n|+i, *}$ and $h_i$ are both zero.

To summarize, we have shown that the generators $y_s$ of the $H_*(B_s)$ shown in bold (namely, $H_*(B_{-|n|}), H_*(B_{-|n|+1}), \dots, H_*(B_{0})$) are nontrivial and distinct in $H_*(\Cone(D_n))$. That is, $F_{X_n, s}$ is nontrivial exactly when $n \leq s \leq 0$, and $\bigoplus_s F_{X_n, s}$ has rank $|n|+1$, as desired.

\eqref{it:nu0ep0} Now suppose that $\nu(K)=0$ and $\varepsilon(K)=0$. Recall that since $\nu(K)=0$, the map $v_{s,*}$ is surjective if and only if $s \geq 0$, and  $h_{s,*}$ is surjective if and only if $s \leq 0$. 

When $s \equiv 0$ mod $n$, we have
\begin{equation}\label{eq:nu0ep0}
\begin{tikzcd}
	& {H_*(A_{-2|n|})} & {H_*(A_{-|n|})} & {H_*(A_0)} & {H_*(A_{|n|})} & {H_*(A_{2|n|})} \\
	\dots &&&&& \dots \\
	& {H_*(B_{-2|n|})} & {\bm{H_*(B_{-|n|})}} & {\bm{H_*(B_0)}} & {H_*(B_{|n|})}
	\arrow["0", from=1-2, to=3-2]
	\arrow[two heads, from=1-3, to=3-2]
	\arrow["0", from=1-3, to=3-3]
	\arrow[two heads, from=1-4, to=3-3]
	\arrow[two heads, from=1-4, to=3-4]
	\arrow["0"', from=1-5, to=3-4]
	\arrow[two heads, from=1-5, to=3-5]
	\arrow["0"', from=1-6, to=3-5]
\end{tikzcd}
\end{equation}
The key difference between \eqref{eq:nu1s0} and  \eqref{eq:nu0ep0} is that now $h_{0,*}$ and $v_{0,*}$ are both surjective. Furthermore, since $\varepsilon(K)=0$, Lemma \ref{lem:HMP} implies that $v_{0,*}$ and $h_{0,*}$ induce the same nontrivial map on homology. In particular, $y_{-|n|}$ and $y_0$ (the generators of $H_*(B_{-|n|})$ and $H_*(B_0)$ respectively) both represent the same nontrivial elements in $H_*(\Cone(D_n))$, since there exists an element $x$ in $H_*(A_0)$ such that $D_{n, *}(x) = y_{-|n|}+y_0$. That is, $F_{-|n|}$ and $F_0$ are both nontrivial, and their images agree.

When $s \not\equiv 0$ mod $n$, the mapping cone is as in \eqref{eq:nu1i}, and the map $F_{X_n, s}$ is nontrivial exactly when $s=n+1, n+2, \dots, -1$, as above.

Thus, $F_{X_n, s}$ is nontrivial exactly when $n \leq s \leq 0$, and $\bigoplus_s F_{X_n, s}$ has rank $|n|$, as desired.

\eqref{it:nu0ep1} The final case to consider is $\nu(K)=0$ and $|\varepsilon(K)|=1$. When $s \equiv 0$ mod $n$, the mapping cone again is an in \eqref{eq:nu0ep0}, the key difference now being that the rank of $D_{n,*}(H_*(A_0))$ has rank 2 instead of rank 1 as in the $\varep(K)=0$ case. In particular, $H_*(B_{-|n|})$ and $H_*(B_0)$ are now both in the image of $D_{n,*}$, and so $F_{n}$ and $F_0$ are both zero. When $s \not\equiv 0$ mod $n$, the mapping cone is as in \eqref{eq:nu1i}, and so the map $F_{X_n, s}$ is nontrivial exactly when $s=n+1, n+2, \dots, -1$, as above. Hence, $F_{X_n, s}$ is nontrivial exactly when $n+1 \leq s \leq -1$, and $\bigoplus_s F_{X_n, s}$ has rank $|n|-1$, as desired.
\end{proof}

With Proposition \ref{prop:nu01} in hand, we are now ready to prove Theorems \ref{thm:trace} and \ref{thm:trace epsilon}. The key point is to consider $\bigoplus_s F_{X_n, s}$ instead of just $F_{X_n, s}$.

\begin{proof}[Proof of Theorem \ref{thm:trace}]
As noted in the introduction, Theorem 1.4 of \cite{HMP} states that if the oriented knot traces $X_n(K)$ and $X_n(K')$ are diffeomorphic, then $\nu(K)=\nu(K')$, except possibly if $n<0$ and $\{ \nu(K), \nu(K')\} = \{0, 1\}$. Thus, in order to prove Theorem \ref{thm:trace}, we need to show that if $n<0$ and $\nu(K)=0$ and $\nu(K')=1$, then $X_n(K)$ and $X_n(K')$ are not diffeomorphic.

The desired result follows from Proposition \ref{prop:nu01}. Indeed, if $n<0$ and $\nu(K)=0$ and $\nu(K')=1$, then the ranks of  $\bigoplus_s F_{X_n(K), s}$ and $\bigoplus_s F_{X_n(K'), s}$ are different; namely, $\bigoplus_s F_{X_n(K), s}$ has rank $|n|$ or $|n|-1$ (depending on $\varep(K)$), while $\bigoplus_s F_{X_n(K'), s}$ has rank $|n|+1$.
\end{proof}

\begin{proof}[Proof of Theorem \ref{thm:trace epsilon}]
We need to show that if $X_n(K)$ and $X_n(K')$ are diffeomorphic, then $|\varep(K)| = |\varep(K')|$. Recall from \cite[Section 3]{Hom-cables} that if $\nu(K) \neq 0$, then $\varep(K) \neq 0$. Thus, by Theorem \ref{thm:trace}, we may assume that $\nu(K)=\nu(K')=0$. Under this assumption, we now need to show that if $\varep(K)=0$ and $|\varep(K')|=1$, then $X_n(K)$ and $X_n(K')$ are not diffeomorphic.

We first consider the case where $n<0$. By Proposition \ref{prop:nu01}, we have that $\bigoplus_s F_{X_n(K), s}$ has rank $|n|$ and $\bigoplus_s F_{X_n(K'), s}$ has rank $|n|-1$.

The case $n>0$ proceeds similarly, by considering $-\big(X_n(K)\big) = X_{-n}(-K)$, and using the fact that $\varep(-K)=-\varep(K)$.

Lastly, we consider $n=0$. We first consider $X_0(K)$, where $\varep(K)=0$. Consider the summand of the mapping cone corresponding to $\s_0$:
\begin{equation}\label{eq:n=0}
\begin{tikzcd}
	{H_*(A_0)} \\
	{} \\
	{H_*(B_0)}
	\arrow["{v_{0,*}+h_{0,*}}", from=1-1, to=3-1]
\end{tikzcd}
\end{equation}
By Lemma \ref{lem:HMP}, we have that $v_{0,*}+h_{0,*}=0$, and hence $F_{X_0(K), 0}$ is injective.

We now consider $X_0(K')$, where $|\varep(K')| = 1$. Since $\nu(K')=0$, we have that $v_{0,*}$ and $h_{0,*}$ are both surjective. However, since $|\varep(K')| = 1$, Lemma \ref{lem:HMP} implies that $v_{0,*}$ and $h_{0,*}$ are distinct. Hence $v_{0,*}+h_{0,*}$ is surjective, and so $F_{X_0(K'), 0}$ is the zero map. Since $F_{X_0(K), 0}$ is injective and $F_{X_0(K'), 0}$ is not, it follows that $X_0(K)$ and $X_0(K')$ are not diffeomorphic. This completes the proof.
\end{proof}

We now prove Theorems \ref{thm:vanishing HF when summing} and \ref{thm:vanishing HF for all spinc}. For our proof strategy, it is convenient to rephrase these theorems as Propositions \ref{prop:nuhat0} and \ref{prop:nuhatneq0} below. Note that part \eqref{it:nuhat0prop1} of both Propositions \ref{prop:nuhat0} and \ref{prop:nuhatneq0} follows from \cite[Proposition 4.2]{HMP}; we include the statements and proofs here for completeness.
Recall that $F_{X_n}$ is the map obtained by summing $F_{X_n,s}$ over all $\spinc$-structures on $X_n$.

\begin{proposition}\label{prop:nuhat0}
If $\nuhat(K) = 0$, then 
\begin{enumerate}
	\item\label{it:nuhat0prop1} the map $F_{X_n,s}$ is trivial for all $s$ if and only if $n \geq 1$, and
	\item\label{it:nuhat0prop2} the map $F_{X_n}$ is trivial if and only if $n=-1$ or $n \geq 1$.
\end{enumerate}
\end{proposition}

\begin{proof}
Suppose that $\nuhat(K) = 0$, in which case $\varep(K) = 0$. By Proposition \ref{prop:nu01}, if $n<0$, then $F_{X_n, s}$ is nonzero for $n \leq s \leq 0$. If $n=0$, then by considering \eqref{eq:n=0} and applying Lemma \ref{lem:HMP}, we have that $F_{X_0, 0}$ is nonzero. Lastly, suppose that $n \geq 1$. Recall that $\varep(K)=0$ implies that $\nu(K)=0$, and
\begin{enumerate}
	\item $v_{s,*}$ is surjective if and only if $s \geq \nu(K)$, and 
	\item $h_{s,*}$ is surjective if and only if $s \leq -\nu(K)$. 
\end{enumerate}
In particular, for $s>0$, we have that $H_*(B_s)$ is in the image of $D_{n,*}(H_*(A_s))$, and for $s<0$, we have that $H_*(B_s)$ is in the image of $D_{n,*}(H_*(A_{s-n}))$.  For $s=0$, we use Lemma \ref{lem:HMP} to conclude that there exists an element $x \in H_*(A_0)$ such that $D_{n,*}(x)= y_0 + y_n$, where $y_0$ and $y_n$ are generators of $H_*(B_0)$ and $H_*(B_n)$ respectively. Hence $y_0$ and $y_n$ are equivalent in $H_*(\Cone(D_{n,*}))$, and we have already concluded that $y_n$ is trivial in $H_*(\Cone(D_{n,*}))$. Hence $F_{X_n,s}$ is zero for all $s$ when $n\geq 1$.

Summing over $\spinc$-structures, we have that $F_{X_n}$ is trivial when $n \geq 1$. Furthermore, we see that $F_{X_n}$ is nontrivial when $n<-1$ or $n=0$; the former since $F_{X_n}$ has nontrivial image in $\HFhat(S^3_n(K), \s_{-1}|_{S^3_n(K)})$, and the latter since $F_{X_0}$ has nontrivial image in $\HFhat(S^3_0(K), \s_0|_{S^3_0(K)})$. Lastly, we consider the case $n=-1$. By Proposition \ref{prop:nu01}, we have that $F_{X_{-1}, s}$ is nonzero if and only if $s\in\{-1, 0\}$ but by the argument immediately following \eqref{eq:nu0ep0}, we have that the images of $F_{X_{-1}, -1}$ and $F_{X_{-1}, 0}$, namely $y_{-1}$ and $y_0$, represent the same element in $H_*(\Cone(D_n))$, implying that when we sum over all $\spinc$-structures, the map is zero.
\end{proof}

\begin{proposition}\label{prop:nuhatneq0}
If $\nuhat(K) \neq 0$, then 
\begin{enumerate}
	\item\label{it:nuhatneq0prop1} the map $F_{X_n,s}$ is trivial for all $s$ if and only if $n \geq \nuhat(K)$, and
	\item\label{it:nuhatneq0prop2} the map $F_{X_n}$ is trivial if and only if $n \geq \nuhat(K)$.
\end{enumerate}
\end{proposition}

\begin{proof}
Throughout, let $\nuhat(K) \neq 0$. 

Suppose that $n < \nuhat(K)$. We will prove that $F_{X_n, \nu(K)-1}$ and $F_{X_n}$ are nontrivial. Consider $H_*(B_{\nu(K)-1})$. We claim that $H_*(B_{\nu(K)-1})$ is not in the image of $D_{n,*}$. It is sufficient to show that neither $v_{\nu(K)-1,*}$ nor $h_{\nu(K)-1-n,*}$ are surjective. Recall that 
\begin{enumerate}
	\item $v_{s,*}$ is surjective if and only if $s \geq \nu(K)$, and 
	\item $h_{s,*}$ is surjective if and only if $s \leq -\nu(K)$. 
\end{enumerate}
Then it is immediate that $v_{\nu(K)-1,*}$ is not surjective. To see that $h_{\nu(K)-1-n,*}$ is not surjective, we observe that by hypothesis, $n < \nuhat(K) = 2\nu(K)-1$, which implies that $-\nu(K) < \nu(K)-1-n$. Hence $h_{\nu(K)-1-n,*}$ is not surjective. Thus, when $\nuhat(K) \neq 0$ and $n < \nuhat(K)$, the map $F_{X_n, \nu(K)-1}$ is nontrivial. Furthermore, since $v_{\nu(K)-1,*}$ and $h_{\nu(K)-1-n,*}$ are both the zero map, the image of $F_{X_n, \nu(K)-1}$, namely $y_{\nu(K)-1}$, is not homologous to the image of any other map $F_{X_n, s}$, where $s \neq \nu(K)-1$. Hence $F_{X_n}$ is nontrivial.

Now suppose that $n \geq \nuhat(K)$. Consider $H_*(B_s)$. If $s \geq\nu(K)$, then $v_{s, *}$ is surjective. Furthermore, since $\nuhat(K) \neq 0$, it follows that $\varep(K) \neq 0$, and so by Lemma \ref{lem:HMP}, we have that $H_*(B_s)$ is in the image of $D_{n,*}(H_*(A_s))$. Hence $F_{X_n, s}$ is trivial for $s \geq\nu(K)$. If $s < \nu(K)$, then we claim that $H_*(B_s)$ is in the image of $D_{n,*}(H_*(A_{s-n}))$. Indeed, the inequalities $s \leq \nu(K)-1$ and $n \geq \nuhat(K) = 2\nu(K)-1$ imply that $s-n \leq -\nu(K)$. Hence $h_{s-n,*}$ is surjective, and then by Lemma \ref{lem:HMP}, we have that $y_{s}$ is in the image of $D_{n,*}$. Thus $F_{X_n,s}$ is trivial for all $s$, as desired.
\end{proof}

\section{Exotic traces}\label{sec:exotictraces}
\subsection{Yasui's family of knot trace pairs}
Yasui \cite{Yasui} considered the two patterns shown in Figure \ref{fig:YasuiPQ}, and proved the following lemma:

\begin{lemma}[{\cite[Lemma 4.3]{Yasui}}]
For any integers $n$ and $m$, and any knot $K$ in $S^3$, the knot traces $X_n(P_{n,m}(K))$ and $X_n(Q_{n,m}(K))$ are homeomorphic.
\end{lemma}
\noindent The proof of the above lemma relies on using handle calculus to show that the manifolds differ by a cork twist.

We will use Theorems \ref{thm:trace} and \ref{thm:trace epsilon} to show that under certain hypotheses, these traces are not diffeomorphic.

\subsection{Tau, epsilon, and nu of Yasui's satellite knots}
We observe that $P_{n,0}(K)$ in Figure \ref{fig:YasuiPQ} is the same knot as $Q^{0,1}_n(K)$ in the notation of \cite{Bodish}; throughout, we will denote this knot by $P_{n,0}(K)$. We summarize the results of \cite[Theorem 1.6, Theorem 1.7, and Section 5]{Bodish} in Table \ref{tab:Bodish}.

\begin{table}[htb!]
\begin{tabular}{c|cccc}
                          & $\varepsilon(K)=-1$ & $\varepsilon(K)=0$                                                & $\varepsilon(K)=1$                                                                           &  \\
                          \hline
$\tau(P_{n,0}(K))$        & $\tau(K)+1$         & \begin{tabular}[c]{@{}l@{}}$\qquad 1, \; n<0 \qquad $\\ $\qquad 0, \; n \geq 0$\end{tabular}  & \begin{tabular}[c]{@{}l@{}}$\tau(K)+1, \; n<2\tau(K)$\\ $\tau(K), \qquad n \geq 2\tau(K)$\end{tabular} &  \\
&&& \\
$\varepsilon(P_{n,0}(K))$ & $1$                 & \begin{tabular}[c]{@{}l@{}}$0, \; n =0$\\ $1, \; n \neq 0$\end{tabular} & $1$  
                                                                                        &  \\
&&& \\
$\nu(P_{n,0}(K))$         & $\tau(K)+1$         & \begin{tabular}[c]{@{}l@{}}$1, \; n<0$\\ $0, \; n \geq 0$\end{tabular}  & \begin{tabular}[c]{@{}l@{}}$\tau(K)+1, \; n<2\tau(K)$\\ $\tau(K), \qquad n \geq 2\tau(K)$\end{tabular} & 
\end{tabular}
\caption{Summary of \cite[Theorem 1.6, Theorem 1.7, and Section 5]{Bodish}.}\label{tab:Bodish}
\end{table}

We can also get bounds on $\tau(P_{n,m}(K))$, as in the following lemma:

\begin{lemma}\label{lem:Pnm}
Let $P_{n,m}$ be the pattern in the left of Figure \ref{fig:YasuiPQ}. 
\begin{enumerate}
	\item If $\varep(K)=0$, $n<0$, and $m\geq 0$, then 
\[ \tau(P_{n,m}(K)) \geq 1. \]
	\item If $\varep(K)=1$, $n < 2\tau(K)$, and $m\geq 0$, then 
\[ \tau(P_{n,m}(K)) \geq \tau(K)+1. \]
\end{enumerate}

\end{lemma}

\begin{proof}
Recall from \cite[Corollary 3]{Livingston-tau} that $\tau$ satisfies a crossing change inequality; namely,  if knots $J_+$ differ $J_-$ by a single crossing change, from positive to negative, then $\tau(J_+)$ is either $\tau(J_-)$ or $\tau(J_-)+1$. Hence $\tau(P_{n,m+1}(K)) \geq \tau(P_{n,m}(K))$. Table \ref{tab:Bodish} tells us that 
\begin{enumerate}
	\item $\tau(P_{n,0}(K)) = 1$ if $\varep(K)=0$ and  $n<0$, and
	\item $\tau(P_{n,0}(K)) = \tau(K)+1$ if $\varep(K)=1$ and  $n<2\tau(K)$. 
\end{enumerate}
This completes the proof.
\end{proof}

We are also interested in $\tau, \varep$, and $\nu$ of $Q_{n,m}(K)$, where $Q_{n,m}$ is the pattern in the right of Figure \ref{fig:YasuiPQ}. Since $Q_{n,m}$ is concordant in $S^1 \times D^{{2}} \times I$ to the identity pattern and $\tau, \nu$, and $\varep$ are all concordance invariants, we have that

\begin{align}\label{eq:tenQ}
\begin{split}
	\tau(Q_{n,m}(K)) &= \tau(K) \\
	\varep(Q_{n,m}(K)) &= \varep(K) \\
	\nu(Q_{n,m}(K)) &= \nu(K). \\
\end{split}
\end{align}
Also recall that $\nu(K) = \tau(K)$ if $\varep(K)=0$ or $1$, and $\nu(K) = \tau(K)+1$ if $\varep(K)=-1$.

We are now ready to prove Theorem \ref{thm:exotic traces}.

\begin{proof}[Proof of Theorem \ref{thm:exotic traces}]
We first consider the case when $\varep(K)=0$, $n<0$, and $m \geq 0$. Then by Lemma \ref{lem:Pnm}, we have that $\tau(P_{n,m}(K)) \geq 1$, which implies that $|\varep(P_{n,m}(K))| = 1$, while \eqref{eq:tenQ} tells us that  $\varep(Q_{n,m}(K))=0$. Since the absolute value of $\varep$ is a trace invariant, it follows that $X_n(P_{n,m}(K))$ and $X_n(Q_{n,m}(K))$ are not diffeomorphic. 

We next consider the case when $\varep(K)=0$ and $n \neq 0$, in which case Table \ref{tab:Bodish} tells us that $\varep(P_{n,0}(K)) = 1$. Again, by \eqref{eq:tenQ}, we have that $\varep(Q_{n,0}(K))=0$, and so it follows that $X_n(P_{n,0}(K))$ and $X_n(Q_{n,0}(K))$ are not diffeomorphic. 

We lastly consider the case when $\varep(K)=1$ and $n<2\tau(K)$, in which case Lemma \ref{lem:Pnm} tells us that for $m \geq 0$, we have $\nu(P_{n,m}(K)) \geq \tau(K)+1$. We have that $\nu(Q_{n,m}(K))=\nu(K)=\tau(K)$, where the first equality follows from \eqref{eq:tenQ} and the second equality follows from the assumption that $\varep(K)=1$. Since $\nu$ is a trace invariant, it follows that $X_n(P_{n,m}(K))$ and $X_n(Q_{n,m}(K))$ are not diffeomorphic.
\end{proof}

\section{Relationship with Baldwin-Sivek's $\widehat{\nu}$}\label{sec:BS}

In this section, we first show that $\widehat{\nu}=|\varep|(2\nu -1)$ agrees with the invariant $\widehat{\nu}$ introduced by Baldwin and Sivek in \cite{BS-instantonconcordance}. To do that, according to \cite[Lemma 10.4]{BS-instantonconcordance}, it suffices to prove the following lemma. 

\begin{lemma}\label{lem:same nu hat}
Let $K$ be a knot in $S^3$. Then
\begin{equation}\label{eq:same nu hat}
  |\varepsilon(K)|(2\nu(K)-1)=\begin{cases}
    \mathrm{max}( 2\nu(K)-1,0), & \nu(K)\geq \nu({-K}).\\
    -\mathrm{max}( 2\nu({-K})-1,0), & \nu(K)\leq \nu({-K}),
  \end{cases}
\end{equation}
where $-K$ is the reverse of the mirror of $K$. In particular, $\widehat{\nu}=|\varep|(2\nu -1)$ agrees with the invariant $\widehat{\nu}$ introduced in \cite{BS-instantonconcordance}.
\end{lemma}

\begin{proof}
    First recall that the standard relations between $\varepsilon$, $\nu$, $\tau$, and the behavior of these invariants under mirroring give us 
\begin{enumerate}
	\item\label{it:epnu0} $\varepsilon(K)=0$ if and only if $\nu(K)=\nu({-K})=0$, and 
	\item\label{it:epnuneq0} if $\varepsilon(K)\neq 0$, then $\nu(K)+\nu({-K})=1$.
\end{enumerate}

    Now we divide the proof into two cases. First, if $\varepsilon(K)=0$, then by property \eqref{it:epnu0} above, both sides of \eqref{eq:same nu hat} are zero, so they agree.
    
    Second, if $|\varepsilon(K)|=1$, then property \eqref{it:epnuneq0} above says $\nu(K)+\nu({-K})=1$ and the left-hand side of \eqref{eq:same nu hat}  equals $2\nu(K)-1$. Let's verify the right-hand side also gives $2\nu(K)-1$. If $\nu(K) \geq \nu({-K})$, then $\nu(K) \geq 1- \nu(K)$, we have $2\nu(K)-1 \geq 0$. Therefore the right-hand side gives $\widehat{\nu}=2\nu-1$. If $\nu(K) \leq \nu({-K})$, then $1-\nu({-K}) \leq \nu({-K})$, we have $2\nu({-K})-1 \geq 0$. Therefore the right-hand side  gives $\widehat{\nu}=-(2\nu({-K})-1)=1-2(1-\nu(K))=2\nu(K)-1$. This finishes the proof.
\end{proof}

The next proposition summarizes how $\widehat{\nu}(K)$ relates to the rank of $\widehat{HF}(S^3_{p/q}(K))$. 

\begin{proposition} \label{prop:dimHF}\cite[Proposition 4.2]{BS-Lspaceknottrace} For any knot $K$ in $S^3$, there is an integer $\widehat{r}_0(K)$ such that \[\operatorname{dim} \widehat{HF}(S^3_{p/q}(K)=q \cdot \widehat{r}_0(K)+|p-q\widehat{\nu}(K)|\] for all nonzero, relatively prime integers $p\neq 0$ and $q>0$.
\end{proposition} 

As a quick consequence of Corollary \ref{cor:nuhat}
and the proposition above, we see that both $\widehat{r}_0(K)$ and $\operatorname{dim}\widehat{HF}(S^3_{p/q}(K))$ are knot trace invariants. To be more precise, we have the following corollary that extends \cite[Theorem 4.6]{BS-Lspaceknottrace} and implies Theorem \ref{thm:dimHF}.

\begin{corollary}
    For any integer $n$, if the oriented knot traces $X_n(K)$ and $X_n(K')$ are diffeomorphic, then $\widehat{r}_0(K)=\widehat{r}_0(K')$. Moreover, $\operatorname{dim} \widehat{HF}(S^3_{p/q}(K)=\operatorname{dim} \widehat{HF}(S^3_{p/q}(K')$ for any nonzero, relatively prime integers $p\neq0$ and $q>0$.
\end{corollary}

\begin{proof}
   If $X_n(K) \cong X_n(K')$, then $\operatorname{dim} \widehat{HF}(S^3_{n}(K))=\operatorname{dim} \widehat{HF}(S^3_{n}(K'))$. Now, the corollary follows from $\widehat{\nu}$ being a knot trace invariant (Corollary \ref{cor:nuhat}) and Proposition \ref{prop:dimHF}.
\end{proof}

\section{Obstructing symplectic fillings}\label{sec:symplecticfillings}

Next we move on to the obstruction of a symplectic structure on knot traces. To prove Theorems \ref{thm:filling for traces} and \ref{thm:filling for 2 handles}, the main tools we use here are the Heegaard Floer contact invariant \cite{HKMcontact} \cite{OScontact} and its naturality properties under strong symplectic fillings \cite{GHIGGINI fillable}. 

We start with the following useful proposition which is essentially a consequence of \cite[Remark 2.14]{GHIGGINI fillable} together with some formal properties of Heegaard Floer homology.

\begin{proposition}\label{prop:non-vanishing map for filling}
    Let $(Y,\xi)$ be a contact $3$-manifold. If the trace $(X,\omega)$ is a strong symplectic filling of $(Y,\xi)$, then the Heegaard Floer map induced by the cobordism \[
F_{{X},\mathfrak{s}} : \widehat{HF}(S^3) \longrightarrow \widehat{HF}(Y)
\] 
is non-vanishing for some $\spinc$ structure $\s$ on $X$.
\end{proposition}

\begin{proof}
    Based on the duality of the cobordism maps between $X_n$ and its upside-down $\overline{X}$ \cite[Section 5.1]{OS4manifold}, it is enough to show the induced map \[F_{\overline{X},\mathfrak{s}} : \widehat{HF}(-Y) \longrightarrow \widehat{HF}(-S^3)
\]
is non-vanishing for some $\spinc$ structure $\s$ on $\overline{X}$. This is easier and more natural to deal with, since it is compatible with the contact invariants.

Suppose $(X,\omega)$ is a strong symplectic filling for $(Y,\xi)$. According to \cite[Proof of Theorem 2.13 and Remark 2.14]{GHIGGINI fillable}, the induced map  on $HF^+$ \[F^+_{\overline{X_n},\mathfrak{s}} \colon HF^+(-Y) \longrightarrow HF^+(-S^3)
\]
maps the plus contact invariant $c^+(\xi)$ to $c^+(\xi_{\std})$, where $\mathfrak{s}$ is the canonical $\spinc$ associated to the symplectic structure, and $\xi_{\std}$ is the standard tight contact structure on $S^3$. In particular, the natural exact triangle that relates the hat and plus versions of Heegaard Floer give us the following commutative diagram: 
\[\begin{tikzcd}
	{\widehat{HF}(-Y) } && {\widehat{HF}(-S^3)} \\
	\\
	{HF^+(-Y) } && {HF^+(-S^3)}
	\arrow["{F_{\overline{X},\mathfrak{s}} }", from=1-1, to=1-3]
	\arrow["{i_{-Y}}"', from=1-1, to=3-1]
	\arrow["i", from=1-3, to=3-3]
	\arrow["{F^+_{\overline{X},\mathfrak{s}}}"', from=3-1, to=3-3]
\end{tikzcd}\]

By the definition of the plus contact invariant, $c^+$ is exactly the image of the hat contact invariant $c$ under the natural inclusion map. Thus, if we start with $c(\xi)\in {\widehat{HF}(-Y) }$ we see $F^+_{\overline{X},\mathfrak{s}}(i_{-Y}(c(\xi)))=F^+_{\overline{X},\mathfrak{s}}(c^+(\xi))=c^+(\xi_{\std})$. In particular, this is a non-trivial map since $c^+(\xi_{\std})$ is non-trivial. The commutativity of the diagram together with the fact that $i$ map from $\widehat{HF}(-S^3)$ to $\widehat{HF}(-S^3)$ is injective imply the top horizontal map $F_{\overline{X},\mathfrak{s}}$ is non-vanishing for some $\spinc$ structure on the cobordism, which finishes the proof.
\end{proof}

Now both Theorem \ref{thm:filling for traces} and \ref{thm:filling for 2 handles} follow easily from the above proposition. 

\begin{proof}[Proof of Theorem \ref{thm:filling for traces}]

Let $X_n$ be a strong symplectic filling of $(S^3_n(K),\xi)$. If $\widehat{\nu}\neq0$, then Theorem \ref{thm:vanishing HF for all spinc} and Proposition \ref{prop:non-vanishing map for filling} imply for any $n\geq \widehat{\nu}(K)= \widehat{\nu}(K)+0^{\widehat{\nu}(K)}$, $X_n$ is not a strong filling. If $\widehat{\nu}=0$, then Theorem \ref{thm:vanishing HF for all spinc} and Proposition \ref{prop:non-vanishing map for filling} imply for any $n\geq 1= \widehat{\nu}(K)+0^{\widehat{\nu}(K)}$, $X_n$ is not a strong filling. This completes the proof.
\end{proof}

\begin{proof}[Proof of Theorem \ref{thm:filling for 2 handles}]

Let $X$ be the corresponding $2$-handle cobordism of $X_{\mathbf{n}}({\mathbf{L}})$ from $S^3$ to $S^3_{\mathbf{n}}(\mathbf{L})$. We start by building $X(\mathbf{L})$ using the $L_i$ handle first. We then have two cobordisms $X_1$ from $S^3$ to $S^3_{n_i}(L_i)$ and $X_2$ from $S^3_{n_i}(L_i)$ to $S^3_{\mathbf{n}}(\mathbf{L})$, where their composite cobordism $X_2 \circ X_1=X$. 

Let $\mathfrak{s}$ be any $\spinc$ structure on $X$. The condition $n_i\neq 0$ implies that $S^3_{n_i}(L_i)$ is a rational homology sphere. Thus, if $\mathfrak{s'}$ is a $\spinc$ structure on $X$ with $\mathfrak{s'}|_{X_1}=\mathfrak{s}|_{X_1}$ and $\mathfrak{s'}|_{X_2}=\mathfrak{s}|_{X_2}$ then we have $\mathfrak{s'}=\mathfrak{s}$. Hence, the composition formula in Heegaard Floer \cite[Theorem 3.4]{OS4manifold} tells us $F_{X_2,\mathfrak{s}|_{X_2}} \circ F_{X_1,\mathfrak{s}|_{X_1}}=F_{X,\mathfrak{s}|_{X}}$.

Now, by Theorem \ref{thm:vanishing HF for all spinc} the condition $n_i \geq \widehat{\nu}(L_i)+0^{\widehat{\nu}(L_i)}$ implies that $F_{X_1,\mathfrak{s}|_{X_1}}$ is always a trivial map, so $F_{X,\mathfrak{s}|_{X}}$ is also always a trivial map for any $\mathfrak{s}$.  Hence by Proposition \ref{prop:non-vanishing map for filling}, $X_n(\mathbf{L})$ cannot be a strong filling.
\end{proof}
\bibliographystyle{alpha}
\bibliography{bib}

\begin{bibdiv}
\begin{biblist}

\bib{Akbulut-exotic}{article}{
      author={Akbulut, Selman},
       title={An exotic {$4$}-manifold},
        date={1991},
        ISSN={0022-040X,1945-743X},
     journal={J. Differential Geom.},
      volume={33},
      number={2},
       pages={357\ndash 361},
         url={http://projecteuclid.org/euclid.jdg/1214446321},
      review={\MR{1094460}},
}

\bib{Akbulut-fake}{article}{
      author={Akbulut, Selman},
       title={A fake compact contractible {$4$}-manifold},
        date={1991},
        ISSN={0022-040X,1945-743X},
     journal={J. Differential Geom.},
      volume={33},
      number={2},
       pages={335\ndash 356},
         url={http://projecteuclid.org/euclid.jdg/1214446320},
      review={\MR{1094459}},
}

\bib{AkbulutLatveyev-exotic}{article}{
      author={Akbulut, Selman},
      author={Matveyev, Rostislav},
       title={Exotic structures and adjunction inequality},
        date={1997},
        ISSN={1300-0098,1303-6149},
     journal={Turkish J. Math.},
      volume={21},
      number={1},
       pages={47\ndash 53},
      review={\MR{1456158}},
}

\bib{Bodish}{misc}{
      author={Bodish, Holt},
       title={Genus, fiberedness, $\tau$ and $\epsilon$ of satellite knots with
  $n$-twisted generalized {M}azur patterns},
        date={2024},
         url={https://arxiv.org/abs/2405.08763},
}

\bib{BS-instantonconcordance}{article}{
      author={Baldwin, John~A.},
      author={Sivek, Steven},
       title={Framed instanton homology and concordance},
        date={2021},
        ISSN={1753-8416,1753-8424},
     journal={J. Topol.},
      volume={14},
      number={4},
       pages={1113\ndash 1175},
         url={https://doi.org/10.1112/topo.12207},
      review={\MR{4332488}},
}

\bib{BS-Lspaceknottrace}{misc}{
      author={Baldwin, John~A.},
      author={Sivek, Steven},
       title={L-spaces and knot traces},
        date={2026},
         url={https://arxiv.org/abs/2501.00914},
}

\bib{FS95}{article}{
      author={Fintushel, Ronald},
      author={Stern, Ronald~J.},
       title={Immersed spheres in {$4$}-manifolds and the immersed {T}hom
  conjecture},
        date={1995},
        ISSN={1300-0098,1303-6149},
     journal={Turkish J. Math.},
      volume={19},
      number={2},
       pages={145\ndash 157},
      review={\MR{1349567}},
}

\bib{GHIGGINIfillable}{article}{
      author={Ghiggini, Paolo},
       title={Strongly fillable contact 3-manifolds without {S}tein fillings},
        date={2005},
        ISSN={1465-3060,1364-0380},
     journal={Geom. Topol.},
      volume={9},
       pages={1677\ndash 1687},
         url={https://doi.org/10.2140/gt.2005.9.1677},
      review={\MR{2175155}},
}

\bib{HANSELMANHF}{misc}{
      author={Hanselman, Jonathan},
       title={Heegaard floer homology and cosmetic surgeries in {$S^3$}},
        date={2020},
         url={https://arxiv.org/abs/1906.06773},
}

\bib{HKL}{article}{
      author={Hom, Jennifer},
      author={Karakurt, \c{C}a\u{g}r\i},
      author={Lidman, Tye},
       title={Surgery obstructions and {H}eegaard {F}loer homology},
        date={2016},
        ISSN={1465-3060,1364-0380},
     journal={Geom. Topol.},
      volume={20},
      number={4},
       pages={2219\ndash 2251},
         url={https://doi.org/10.2140/gt.2016.20.2219},
      review={\MR{3548466}},
}

\bib{HKMcontact}{article}{
      author={Honda, Ko},
      author={Kazez, William~H.},
      author={Mati\'c, Gordana},
       title={On the contact class in {H}eegaard {F}loer homology},
        date={2009},
        ISSN={0022-040X,1945-743X},
     journal={J. Differential Geom.},
      volume={83},
      number={2},
       pages={289\ndash 311},
         url={http://projecteuclid.org/euclid.jdg/1261495333},
      review={\MR{2577470}},
}

\bib{HomLidman19}{article}{
      author={Hom, Jennifer},
      author={Lidman, Tye},
       title={A note on positive-definite, symplectic four-manifolds},
        date={2019},
        ISSN={1435-9855,1435-9863},
     journal={J. Eur. Math. Soc. (JEMS)},
      volume={21},
      number={1},
       pages={257\ndash 270},
         url={https://doi.org/10.4171/JEMS/835},
      review={\MR{3880209}},
}

\bib{HMP}{article}{
      author={Hayden, Kyle},
      author={Mark, Thomas~E.},
      author={Piccirillo, Lisa},
       title={Exotic {M}azur manifolds and knot trace invariants},
        date={2021},
        ISSN={0001-8708,1090-2082},
     journal={Adv. Math.},
      volume={391},
       pages={Paper No. 107994, 30},
         url={https://doi.org/10.1016/j.aim.2021.107994},
      review={\MR{4317407}},
}

\bib{Hom-cables}{article}{
      author={Hom, Jennifer},
       title={Bordered {H}eegaard {F}loer homology and the tau-invariant of
  cable knots},
        date={2014},
        ISSN={1753-8416,1753-8424},
     journal={J. Topol.},
      volume={7},
      number={2},
       pages={287\ndash 326},
         url={https://doi.org/10.1112/jtopol/jtt030},
      review={\MR{3217622}},
}

\bib{Hom-conc}{article}{
      author={Hom, Jennifer},
       title={The knot {F}loer complex and the smooth concordance group},
        date={2014},
        ISSN={0010-2571,1420-8946},
     journal={Comment. Math. Helv.},
      volume={89},
      number={3},
       pages={537\ndash 570},
         url={https://doi.org/10.4171/CMH/326},
      review={\MR{3260841}},
}

\bib{HomWu}{article}{
      author={Hom, Jennifer},
      author={Wu, Zhongtao},
       title={Four-ball genus bounds and a refinement of the
  {O}zsv\'ath-{S}zab\'o{} tau invariant},
        date={2016},
        ISSN={1527-5256,1540-2347},
     journal={J. Symplectic Geom.},
      volume={14},
      number={1},
       pages={305\ndash 323},
         url={https://doi.org/10.4310/JSG.2016.v14.n1.a12},
      review={\MR{3523259}},
}

\bib{Livingston-tau}{article}{
      author={Livingston, Charles},
       title={Computations of the {O}zsv\'ath-{S}zab\'o{} knot concordance
  invariant},
        date={2004},
        ISSN={1465-3060,1364-0380},
     journal={Geom. Topol.},
      volume={8},
       pages={735\ndash 742},
         url={https://doi.org/10.2140/gt.2004.8.735},
      review={\MR{2057779}},
}

\bib{LM98}{incollection}{
      author={Lisca, P.},
      author={Mati\'c, G.},
       title={Stein {$4$}-manifolds with boundary and contact structures},
        date={1998},
      volume={88},
       pages={55\ndash 66},
         url={https://doi.org/10.1016/S0166-8641(97)00198-3},
        note={Symplectic, contact and low-dimensional topology (Athens, GA,
  1996)},
      review={\MR{1634563}},
}

\bib{MPV20}{article}{
      author={Mark, Thomas~E.},
      author={Piccirillo, Lisa},
      author={Vafaee, Faramarz},
       title={On the {S}tein framing number of a knot},
        date={2020},
        ISSN={1527-5256,1540-2347},
     journal={J. Symplectic Geom.},
      volume={18},
      number={1},
       pages={191\ndash 215},
         url={https://doi.org/10.4310/jsg.2020.v18.n1.a5},
      review={\MR{4088751}},
}

\bib{OO09}{article}{
      author={Ohta, Hiroshi},
      author={Ono, Kaoru},
       title={Simple singularities and topology of symplectically filling
  {$4$}-manifold},
        date={1999},
        ISSN={0010-2571,1420-8946},
     journal={Comment. Math. Helv.},
      volume={74},
      number={4},
       pages={575\ndash 590},
         url={https://doi.org/10.1007/s000140050106},
      review={\MR{1730658}},
}

\bib{OS00}{article}{
      author={Ozsv\'ath, Peter},
      author={Szab\'o, Zolt\'an},
       title={The symplectic {T}hom conjecture},
        date={2000},
        ISSN={0003-486X,1939-8980},
     journal={Ann. of Math. (2)},
      volume={151},
      number={1},
       pages={93\ndash 124},
         url={https://doi.org/10.2307/121113},
      review={\MR{1745017}},
}

\bib{OS-4ball}{article}{
      author={Ozsv\'ath, Peter},
      author={Szab\'o, Zolt\'an},
       title={Knot {F}loer homology and the four-ball genus},
        date={2003},
        ISSN={1465-3060,1364-0380},
     journal={Geom. Topol.},
      volume={7},
       pages={615\ndash 639},
         url={https://doi.org/10.2140/gt.2003.7.615},
      review={\MR{2026543}},
}

\bib{OS-knots}{article}{
      author={Ozsv\'ath, Peter},
      author={Szab\'o, Zolt\'an},
       title={Holomorphic disks and knot invariants},
        date={2004},
        ISSN={0001-8708,1090-2082},
     journal={Adv. Math.},
      volume={186},
      number={1},
       pages={58\ndash 116},
         url={https://doi.org/10.1016/j.aim.2003.05.001},
      review={\MR{2065507}},
}

\bib{OS-3mfds}{article}{
      author={Ozsv\'ath, Peter},
      author={Szab\'o, Zolt\'an},
       title={Holomorphic disks and topological invariants for closed
  three-manifolds},
        date={2004},
        ISSN={0003-486X,1939-8980},
     journal={Ann. of Math. (2)},
      volume={159},
      number={3},
       pages={1027\ndash 1158},
         url={https://doi.org/10.4007/annals.2004.159.1027},
      review={\MR{2113019}},
}

\bib{OScontact}{article}{
      author={Ozsv\'ath, Peter},
      author={Szab\'o, Zolt\'an},
       title={Heegaard {F}loer homology and contact structures},
        date={2005},
        ISSN={0012-7094,1547-7398},
     journal={Duke Math. J.},
      volume={129},
      number={1},
       pages={39\ndash 61},
         url={https://doi.org/10.1215/S0012-7094-04-12912-4},
      review={\MR{2153455}},
}

\bib{OS4manifold}{article}{
      author={Ozsv\'ath, Peter},
      author={Szab\'o, Zolt\'an},
       title={Holomorphic triangles and invariants for smooth four-manifolds},
        date={2006},
        ISSN={0001-8708,1090-2082},
     journal={Adv. Math.},
      volume={202},
      number={2},
       pages={326\ndash 400},
         url={https://doi.org/10.1016/j.aim.2005.03.014},
      review={\MR{2222356}},
}

\bib{OS-integer}{article}{
      author={Ozsv\'ath, Peter~S.},
      author={Szab\'o, Zolt\'an},
       title={Knot {F}loer homology and integer surgeries},
        date={2008},
        ISSN={1472-2747,1472-2739},
     journal={Algebr. Geom. Topol.},
      volume={8},
      number={1},
       pages={101\ndash 153},
         url={https://doi.org/10.2140/agt.2008.8.101},
      review={\MR{2377279}},
}

\bib{OS-rational}{article}{
      author={Ozsv\'ath, Peter~S.},
      author={Szab\'o, Zolt\'an},
       title={Knot {F}loer homology and rational surgeries},
        date={2011},
        ISSN={1472-2747,1472-2739},
     journal={Algebr. Geom. Topol.},
      volume={11},
      number={1},
       pages={1\ndash 68},
         url={https://doi.org/10.2140/agt.2011.11.1},
      review={\MR{2764036}},
}

\bib{Plamenevskaya04}{article}{
      author={Plamenevskaya, Olga},
       title={Bounds for the {T}hurston-{B}ennequin number from {F}loer
  homology},
        date={2004},
        ISSN={1472-2747,1472-2739},
     journal={Algebr. Geom. Topol.},
      volume={4},
       pages={399\ndash 406},
         url={https://doi.org/10.2140/agt.2004.4.399},
      review={\MR{2077671}},
}

\bib{Rasmussen-s}{article}{
      author={Rasmussen, Jacob},
       title={Khovanov homology and the slice genus},
        date={2010},
        ISSN={0020-9910,1432-1297},
     journal={Invent. Math.},
      volume={182},
      number={2},
       pages={419\ndash 447},
         url={https://doi.org/10.1007/s00222-010-0275-6},
      review={\MR{2729272}},
}

\bib{RenWillis}{misc}{
      author={Ren, Qiuyu},
      author={Willis, Michael},
       title={Khovanov homology and exotic $4$-manifolds},
        date={2025},
         url={https://arxiv.org/abs/2402.10452},
}

\bib{Yasui}{article}{
      author={Yasui, Kouichi},
       title={Corks, exotic 4-manifolds and knot concordance},
        date={2026},
        ISSN={0022-040X,1945-743X},
     journal={J. Differential Geom.},
      volume={132},
      number={3},
       pages={547\ndash 576},
         url={https://doi.org/10.4310/jdg/1770827024},
      review={\MR{5030469}},
}

\end{biblist}
\end{bibdiv}

\end{document}